\documentclass[11pt]{amsart}
\usepackage{hyperref, color}
\usepackage[english]{babel}
\usepackage{amsfonts, amsmath, amsthm, amssymb,amscd,indentfirst}
\usepackage[percent]{overpic}  
\usepackage{esint}
\usepackage{graphicx}
\usepackage{epstopdf}
\usepackage{amsmath,amssymb,latexsym,indentfirst}
\usepackage[latin1]{inputenc}

\newtheorem{theorem}{Theorem}[section]

\newtheorem{proposition}{Proposition}[section]

\newtheorem{corollary}{Corollary}[section]

\newtheorem{remark}{Remark}[section]

\newcommand{\field}[1]{\mathbb{#1}}
\newcommand{\real}{\field{R}}

\begin{document}

\title[The fundamental tone of the $p$-Laplacian on  manifolds]
{On the fundamental tone of the $p$-Laplacian on Riemannian manifolds and applications}

\author[Carvalho]{Francisco G. de S. Carvalho}
\address{ Universidade Federal do Piau\'i,
Picos, PI, Brazil}  \email{franciscogsc.mat@ufpi.edu.br}
\author[Cavalcante]{Marcos P. de A. Cavalcante} 
\address{Universidade Federal de Alagoas,
Macei\'o, AL, Brazil} 
\email{marcos@pos.mat.ufal.br}

%    General info
\date{\today}


\begin{abstract}

We present a general lower bound for the fundamental tone for the $p$-Laplacian on 
Riemannian manifolds carrying a special kind of function. 
We then apply our result to the cases of 
negatively curved simply connected manifolds,  a class of warped product manifolds and for a class of 
Riemannian submersions.
\end{abstract}

\date{\today}

\maketitle

\section{Introduction}\label{intro}

%
%
%In ... Mckean proved that if $\Omega$ is a bounded domain in the hyperbolic space
%$\mathbb H^n$, then the first eigenvalue of the Dirichelet problem is bounded from below
%\[
%\lambda_1(\Omega)\geq \frac{(n-1)^2}{4}.
%\]
%More recently Savo and Artanoshing proved some refined estimates in the case of
%geodesic balls in the hyperbolic space. Precisely, if $B_R$ denotes a geodesic ball
%of radius $R>0$ in $\mathbb H^n$ then
%\[
%\lambda_1(B_R)\geq \frac{(n-1)^2}{4} + \frac{\pi^2}{R^2}
%\]

Let $\Omega$ be a bounded domain in a $n$-dimensional Riemannian manifold $M$, $n\geq 2$. 
The first eigenvalue of the Dirichlet problem for the  Laplace-Beltrami operator on $\Omega$ can be characterized variationally as
\[
\lambda_ 1(\Omega) = \inf \bigg\{\frac{\int_\Omega \|\nabla u\|^2 \, dM}{\int_\Omega u^2\, dM}:
u\in W^{1,2}_0(\Omega), u\neq 0 \bigg \},
\]
where $W^{1,2}_0(\Omega)$ denotes the Sobolev space with trace in $\partial \Omega$.
If $M$ is not compact,  \emph{the first eigenvalue} of $M$ is defined as the limit
\[
\lambda_1(M)=\lim_{k\to \infty}\lambda_1(\Omega_k),
\]
if it exists, where $\Omega_1\subset \Omega_2\subset\cdots$ is an exhaustion of $M$. 
It is easy to  see that this limit does not
depend on the exhaustion we choose. 
This number is a very important geometrical invariant, and is 
also called \emph{the fundamental tone of $M$} by some authors.
A classical problem in spectral geometry is to find conditions on $M$ which imply
$\lambda_1(M)>0$ (see \cite[Chap. III]{SY}). For instance, if $M$ has
infinite volume, the positivity of  $\lambda_1(M)$ implies  the manifold  $M$ is hyperbolic 
(or non-parabolic) (see \cite[Proposition 10.1]{G}), that is,  $M$ admits  a positive Green function.

In this sense, we recall the classical paper of  McKean \cite{McKean}, which proved that if
$M$ is a simply connected Riemannian manifold with sectional curvature $K_M\leq -\kappa^2<0$, then
\[
\lambda_1(M) \geq \frac{(n-1)^2\kappa^2}{4}.
\] 
Notice that this lower bound is precisely the first eigenvalue of $\mathbb H^n(-\kappa^2)$, 
the hyperbolic space of constant  negative curvature $-\kappa^2$. 
This result of McKean was then extended and used by many authors.
In particular, Veeravalli \cite{Veeravalli}, using very simple ideas, proved a positive lower bound estimate for
$\lambda_1(M)$ for a large class of Riemannian manifolds given by warped products.

In this paper, we generalize the work of Veeravalli for the $p$-Laplacian operator on Riemannian manifolds,
and we present some applications.
Recall that the $p$-Laplacian $\Delta_p$, $1<p<\infty$, is the nonlinear operator defined as
\[
\Delta_p u = \textrm{div}\, (\|\nabla u\|^{p-2}\nabla u)
\]
for smooth functions on $M$. The $p$-Laplacian shares many  properties with the Laplace-Beltrami 
operator ($p=2)$, and it appears naturally in many  problems in Physics and Applied Mathematics
(see for instance \cite{AP, MR1312235,MR2164332} and the references cited therein). 
Although it is a nonlinear operaror if $p\neq 2$, its first eigenvalue is also characterized variationally  in a similar way as
\[
\lambda_ {1,p}(\Omega) = \inf \bigg\{\frac{\int_\Omega \|\nabla u\|^p \, dM}{\int_\Omega u^p\, dM}:
u\in W^{1,p}_0(\Omega), u\neq 0 \bigg \}.
\]

As before, if $M$ is not compact, we define the fundamental tone for the $p$-Laplacian
on $M$, $\lambda_{1,p}(M)$,  as the limit of the first eigenvalue for some exhaustion by compacts.
Our first result is a criteria to obtain a positive lower bound for $\lambda_ {1,p}$ 
on bounded domains carrying 
a special kind of function. More precisely, we prove:
\begin{theorem}\label{teo1}
Let $\Omega$ be a bounded domain on a Riemannian manifold $M$, 
and assume that there exist a smooth function
$f:\Omega\to \real$ satisfying that $\|\nabla f\|\leq a$ and $\Delta_p f\geq b$ for some constants $a,b>0$. 
Then the first eigenvalue
of the $p$-Laplacian satisfies 
\[
\lambda_{1,p}(\Omega)\geq \frac{b^p}{p^pa^{p(p-1)}}.
\]
\end{theorem}

Similarly to the linear case ($p=2$), the positivity of the fundamental tone of $\Delta_p$
on a complete manifold with infinite volume implies the $p$-hiperbolicity, that is, the existence of
a positive Green function for the $p$-Laplacian on $M$. See for instance the work of the Batista, 
Santos, and the second named author in \cite{BCS}. 

As a first application of our general estimate we obtain in a simple way the 
following generalization of McKean's theorem: 
\begin{corollary}\label{c1}
Let $M^n$ be a complete simply connected Riemannian manifold  such that the sectional curvature satisfies 
$K_M\leq -\kappa^2$. If $\Omega\subset M$ is a bounded domain, then
\[
\lambda_{1,p}(\Omega) \geq \frac{(n-1)^p\kappa^p}{p^p}\coth^pR,
\]
where $R>0$ is such that $\Omega$ is contained in the geodesic ball of radius $R$.

In particular, $\lambda_{1,p}(M)\geq  \frac{(n-1)^p\kappa^p}{p^p} =\lambda_{1,p}(\mathbb H^n(-\kappa^2)) .$
\end{corollary}

We point out the same result was obtained by Poliquin in \cite{Poliquin} using  estimates by the Cheeger constant. In fact, in \cite[Proposition 4.5]{Poliquin} it is proved that 
$\lambda_{1,p}(\Omega)\geq  \frac{(n-1)^p\kappa^p}{p^p}$ for all domain $\Omega \subset M$.

\begin{remark}
To compute the exact value of $\lambda_p(\Omega)$ for compact domains on Riemannian manifolds 
is a hard problem. 
For instance, for $p=2$ and geodesic balls, it was computed by Savo  \cite{MR2480663} 
in the 3-dimensional hyperbolic space  $\mathbb H^3$, 
and by Kristaly \cite{Kristaly} for odd-dimensional hyperbolic spaces.
On the other hand, this kind of estimate  as in Corollary \ref{c1} is particularly useful in order to prove a version of 
Berestycki-Caffarelli-Nirenberg conjecture for 
domains in the hyperbolic plane (see \cite{CC,EM}).
\end{remark}

Next,  inspired in the work of Veeravally \cite{Veeravalli} we present an estimate for a class of warped
product metrics. 

\begin{corollary}\label{cor}
Let $(N^{n-1}, g_0)$ be a  Riemannian manifold and consider $M^n = \real\times N$
endowed with the warped metric $ds^2 = dt^2 + e^{2\rho(t)}g_0$, 
where the warped function  satisfies $\rho'(t)\geq \kappa>0$, for some constant $\kappa$. 
Then, the fundamental tone  of  the $p$-Laplacian of $M$ is bounded from below by 
\[
\lambda_{1,p}(M) \geq \frac{(n-1)^p}{p^p}\kappa^p.
\]
\end{corollary}
This class of warped manifolds contains the hyperbolic space, and so the estimate is sharp.
On the other hand, 
there are also some examples of Riemannian manifolds in this class whose sectional curvature
is positive  in some directions, as we can
see if we choose, for instance, $N=\mathbb S^{n-1}$, the round sphere, and $\rho (t)=\kappa t$.

As observed   in \cite{CM},  this kind of estimate
presented in Corollary \ref{cor} can be \emph{lifted} 
for Riemannian manifolds that admit a Riemannian submersion over the hyperbolic space 
whose fibers have uniformly  bounded mean curvature. 
In this context, we have the following theorem for the $p$-Laplacian

\begin{theorem}\label{sub} Let $\widetilde M^m$ be  a complete Riemannian manifold that admits 
a Riemannian submersion over $M^n$, where $M^n$ is given in Corollary \ref{cor}. If the mean
curvature of the fibers satisfy $\|H^\mathcal F\|\leq \alpha$ for some $\alpha < (n-1)\kappa^{1/p}$
then 
\[
\lambda_{1,p}(M) \geq \frac{\big((n-1)^p\kappa - \alpha\big)^p}{p^p}.
\]
\end{theorem}
We notice that this is the dual version of similar results for isometric immersions as we can find  
in \cite{BCC, CL, Veeravalli} for $p=2$, and \cite{DuMao, ES, MR4089470}, for  $p>1$.
We point out that in \cite{ES} and in \cite{MR4089470} they used the isoperimetric constant generated
by vector fields presented in \cite{BM}, while in \cite{DuMao} they used the distance function restricted to the submanifold in order to obtain  the estimate.

Recently, Polymerakis in \cite{Polymerakis} established a new estimate for the fundamental  tone
of the Laplace-Beltrami operator of the total space of 
Riemanniann submersions in terms of the spectrum of the base space. 
It is an interesting problem to find a similar estimate for the $p$-Laplacian operator.

This paper is organized as follows. In Section \ref{general} we present the proof of Theorem \ref{teo1}
and some direct consequences. In Section \ref{app} we present the basic material to prove Corollary \ref{c1},
Corollary \ref{cor}, and Theorem \ref{sub}.

\subsection*{Acknowledgments} 
The second author was partially supported by  
Brazilian National Council for Scientific and Technological Development  (CNPq Grant 309733 / 2019-7) and 
Coordenação de Aperfeiçoamento de Pessoal de Nível Superior - Brasil 
(CAPES-COFECUB 88887.143161/ 2017-0 and
CAPES-MATHAMSUD 88887. 368700/ 2019-00).

\section{A general estimate}\label{general}

Given $p>1$, let us consider
the Dirichlet eigenvalue problem associated to the $p$-Laplacian operator on a bounded domain 
$\Omega \subset M$, that is,
\begin{equation}\label{p1}
\begin{cases}
\Delta_p u +\lambda |u|^{p-2}u = 0 \textrm{ in } \Omega, \\
u=0,\textrm { on } \partial \Omega.
\end{cases}
\end{equation}

It is well known that the first eigenvalue of problem (\ref{p1}) can be characterized
variationally as
\[
\lambda_ {1,p}(\Omega) = \inf \bigg\{\frac{\int_\Omega \|\nabla u\|^p \, dM}{\int_\Omega u^p\, dM}:
u\in W^{1,p}_0(\Omega), u\neq 0 \bigg \}.
\]
%Moreover, $\lambda_{1,p}(\Omega)$ is simples, in the sense that two eigenfunctions associated
%to $\lambda_{1,p}(\Omega)$ coincide up a constant. See [?????]. 

In this paper, all integrations are made with respect to the volume form given by the Riemannian metric.
For the sake of simplicity, in the sequel we will omit the volume form $dM$ in the integrals.  
Now we restate and prove our fist result below.

%\begin{lemma} 
%The set $\sigma_p(M)$ of eigenvalues of problem (\ref{p1}) is an unbounded set of $(0,+\infty)$
%\end{lemma}
%
%
%\begin{lemma} The infimum $\lambda_1 = \inf \sigma_p(M)$ is positive and it is itself an eigenvalue.
%\end{lemma}
%
%
%\begin{lemma} The first eigenfunction can be chosen positive.
%\end{lemma}
%
%
%\begin{lemma} If $\Omega$ is a geodesic ball, then the first eigenfunction is radial.
%\end{lemma}
%
%
%\section{Basic facts}
%
%
%Inspired in the work of Veeravalli in \cite{MR2016701} we have. 
%
%
%\begin{lemma}
%Let $\Omega$ be a bounded domain in $M$ e assume that there exists a smooth function
%$f:\Omega\to \real$ such that $|\nabla f|\leq a$ and $\Delta f\geq b$. Then
%\[
%\lambda_1(\Omega)\geq \frac{b^2}{4a^2}.
%\]
%\end{lemma}
%\begin{proof}
%Given any $h\in C^\infty_0(\Omega)$ we have
%\begin{eqnarray*}
%b\int_\Omega h^2&\leq &\int_\Omega h^2\Delta f = - \int_\Omega \langle \nabla h^2, \nabla f\rangle\\
%&=& -2\int_\Omega h\langle \nabla h, \nabla f \rangle \leq 2a\int_\Omega |h||\nabla h|\\
%&\leq& \frac b 2\int_\Omega h^2 + \frac{2a^2}{b}\int_\Omega|\nabla h|^2
%\end{eqnarray*}
%So,
%\[
%\int_\Omega |\nabla h^2| \geq  \frac{b^2}{4a^2}\int_\Omega h^2
%\]
%and the result follows.
%
%
%\end{proof}
%\begin{corollary}
%Let $\Omega$ be a bounded domain in $\mathbb H^n$  contained in a geodesic ball $B_R$. 
%Then 
%\[
%\lambda_1(\Omega)\geq \frac{(n-1)^2}{4}\cosh^2 R
%\]
%\end{corollary}
%\begin{proof}
%Take $f$ the distance function to the center of $B_R$. Then $|\nabla f|= 1$
%and $\Delta f = (n-1) \coth f \geq (n-1)\coth R$.
%\end{proof}
%
%
%\bigskip

%

\begin{theorem}%\label{teo1}
Let $\Omega$ be a bounded domain on a Riemannian manifold $M$, 
and assume that there exist a smooth function
$f:\Omega\to \real$ satisfying that $\|\nabla f\|\leq a$ and $\Delta_p f\geq b$ for some constants $a,b>0$. 
Then the first eigenvalue
of the $p$-Laplacian satisfies 
\[
\lambda_{1,p}(\Omega)\geq \frac{b^p}{p^pa^{p(p-1)}}.
\]
\end{theorem}

\begin{proof}
We first note that by density we can use smooth functions in the variational characterization 
of $\lambda_{1,p}(\Omega)$.  In particular, given $v\in C^\infty_0(\Omega)$ we have

\begin{eqnarray*}
b\int_\Omega |v|^p&\leq &\int_\Omega |v|^p\Delta_p f 
= - \int_\Omega \langle \nabla |v|^p, \|\nabla f\|^{p-2}\nabla f\rangle\\
&=& -p\int_\Omega |v|^{p-1}\langle \nabla |v|,  \|\nabla f\|^{p-2}\nabla f \rangle 
\leq p\int_\Omega |v|^{p-1}\|\nabla v\| \|\nabla f\|^{p-1}\\
&\leq& p\int_\Omega |v|^{p-1}a^{p-1}\|\nabla v\|.
%&\leq& \frac b 2\int_\Omega h^2 + \frac{2a^2}{b}\int_\Omega|\nabla h|^2
\end{eqnarray*}

Now we wish to apply Young inequality in the last term. To do that, 
let us introduce a constant $\theta>0$ to be chosen latter. So, if $q$ is the conjugate of $p$ we get

\begin{eqnarray*}
|v|^{p-1}a^{p-1}\|\nabla v\| &\leq& \frac{\theta^q|v|^{q(p-1)}}{q} + \frac{a^{p(p-1)}\|\nabla v\|^p}{p \theta^p} \\
&=&  \frac{(p-1)\theta^{p/(p-1)}|v|^{p}}{p} + \frac{a^{p(p-1)}\|\nabla v\|^p}{p\theta^p}.
\end{eqnarray*}
That is,
\[
 p|v|^{p-1}a^{p-1}\|\nabla v\| \leq  (p-1)\theta^{p/(p-1)}|v|^{p} + \frac{a^{p(p-1)}\|\nabla v\|^p}{\theta^p} .
\]

Now, let us choose $\theta$ such that $b-(p-1)\theta^{p/(p-1)} = \frac b p$,
that is, $ \theta ^p = \frac {b^{p-1}}{p^{p-1}}$. Hence we get
\[
\frac b p \int_\Omega |v|^p\leq \frac{p^{p-1}a^{p(p-1)}}{b^{p-1}}\int_\Omega \|\nabla v\|^p.
\]
Rearranging the terms we obtain that
\[
\int_\Omega \|\nabla v\|^p \geq \frac{b^p}{p^pa^{p(p-1)}} \int_\Omega |v|^p,
\]
and this concludes the proof.
\end{proof}

As a first application, we apply Theorem \ref{teo1} in the case of domains in
space forms using the distance function.
% This  result will be used in a forthcoming paper \cite{CC}  by the authors.

Let us denote by $M^n(c)$ the simply connected space form of constant sectional curvature $c$. 
In particular, if $c<0$ we write $c = -\kappa^2$ and $M^n(-\kappa^2)=\mathbb H^n(-\kappa^2)$ is the hyperbolic
space, if $c=0$, $M^n(0)=\real^n$ is the Euclidean space, and if $c=\kappa^2>0$, $M^n(\kappa^2)
=\mathbb S^n(\kappa^2)$ is the sphere of radius $\frac 1 \kappa$

\begin{corollary}\label{cor1}
Let $\Omega$ be a bounded domain in the space form $M^n(c)$.
If $\Omega$ is contained in a geodesic ball $B_R$, then 
\begin{equation*}	
    \begin{array}{lll}
\medskip   \lambda_{1,p}(\Omega)\geq \frac{(n-1)^p (\sqrt{-c})^{p}}{p^p}\coth^p (\sqrt{-c}\, R), & \textrm{if}&c<0, \\ 
\medskip    \lambda_{1,p}(\Omega)\geq\frac{(n-1)^p}{p^pR^p},      &  \textrm{if}&c=0, \\ 
  \lambda_{1,p}(\Omega)\geq  \frac{(n-1)^p (\sqrt c)^p}{p^p}\cot^p (\sqrt c \, R), &  \textrm{if}&c>0.
    \end{array}
\end{equation*}
\end{corollary}
\begin{proof}
The metric in $M^n(c)$ in polar coordinates is given by 
$
g = dr^2+f_c^2(r)d\omega^2,
$
where $d\omega^2$ is the standard metric on $\mathbb S^{n-1}$ and 
\begin{equation*}
\displaystyle  f_c(r)=\left\{
\begin{array}{rll}
 \frac{1}{\sqrt{-c}}\sinh({\sqrt{-c}}\, r), & &\textrm{if } c<0, \\
  r,       & & \textrm{if } c=0,\\
 \frac{1}{\sqrt{c}}\sin({\sqrt{c}}\, r),,& & \textrm{if } c>0.
\end{array}
\right.
\end{equation*}
If we denote by $r_c$ is the distance function to the center of $B_R$, we know that 
 $|\nabla r_c|= 1$ and $\Delta_p r_c = \Delta r_c = (n-1) f'_c/f_c$. In particular,
\begin{equation*}
\displaystyle \Delta_p r_c \geq \left\{
\begin{array}{rll}
\medskip (n-1)\sqrt{-c}\coth(\sqrt{-c}\,R), & &\textrm{if } c<0, \\
\medskip (n-1)/ R,       & & \textrm{if } c=0,\\
  \medskip  (n-1) \sqrt c \cot(\sqrt c \, R),& & \textrm{if } c>0,
\end{array}
\right.
\end{equation*}
and the result follows directly from Theorem \ref{teo1}.
\end{proof}

If $c\leq0$,  we know that $M^n(c)$  has empty cut locus and there are geodesic balls $B_R(p)$ 
for any $R>0$. In this case we define the function
\[
L: [0, \infty)\to \big(\frac{(n-1)^p (\sqrt{-c})^{p}}{p^p},\infty\big)
\] 
as $L(R) = \lambda_{1,p}(B_R)$. Note that, from the definition of $\lambda_{1,p}$, $L$ is 
a decreasing function. 
On the other hand, since $L$ is a continuous function (see for instance
\cite{GL}, \cite{MR2164332}), and  from Corollary \ref{cor1} above we obtain 

\begin{proposition} Given any $\lambda >\frac{(n-1)^p}{p^p}$
there exists unique $R_\lambda>0$
such that first eigenvalue of the  Dirichlet problem for the $p$-Laplacian on the
geodesic ball $B_{R_\lambda}\subset \mathbb H^n$
with radius $R_\lambda$  is precisely $\lambda.$
\end{proposition}

\section{Applications}\label{app}

\subsection{Negatively curved manifolds} In this subsection we prove Corollary \ref{c1}. 
From the variational characterization of $\lambda_{1,p}(\Omega)$, we  observe that
if $\Omega_1\subset \Omega_2$, then $\lambda_{1,p}(\Omega_1)\geq\lambda_{1,p}(\Omega_2)$.
In particular, we just need to present an estimate for geodesic balls.

Fixed a geodesic ray  $\gamma$ on $M$,  we consider the Busemann function $B$ 
associated  to $\gamma$. More precisely, $B:M\to \real$ is given by
\[
B(x) = \lim_{t\to \infty} [r(x,\gamma(t)) - t],
\] 
where $r(x, \gamma(t))$ denotes the distance between $x$ and $\gamma (t)$
(see \cite{MR3365851} for details).  
Recall that in nonpositive curvature Busemann functions are of class $C^2(M)$ 
(see \cite{MR512919}) and $\|\nabla B\| =1$. Moreover, from the proof of Lemma 2.3 in
\cite{MR3365851} we have $\Delta_pB = \Delta B \geq (n-1 )\kappa\coth R $ on each geodesic ball
of radius $R>0$. Now, the proof of Corollary \ref{c1} follows directly from Theorem \ref{teo1}.

\subsection{A class of warped products} As stated in the Introduction, let us consider 
$M= \real \times N^{n-1}$, where $(N^{n-1},g_0)$ is an arbitrary Riemannian manifold, and $M$ is
endowed with the warped metric $g = dt^2 + e^{2\rho(t)}g_0$. Following \cite{BCC, CM},
we consider $F:M\to \real$, $F(s,x) = s$. This is the Busemann function associated to the
geodesic ray corresponding to the fact $\real$. A direct computation gives $|\nabla F| = 1$ and 
$$\Delta_p F = \Delta F = (n-1)\rho'(t).$$
So, assuming that $\rho'(t)\geq\kappa>0$ we obtain the claim in Corollary \ref{cor}.

\subsection{Riemannian submersions}
Let $\widetilde M^m$ and $M^n$ be Riemannian manifolds with $m\geq n$ and let
$\pi: \widetilde M^m\to M^n$ be a Riemannian submersion, that is, $\pi$ is a surjective map,
its differential $d\pi_p:T_p \widetilde M\to T_x M$ is surjective for all $p\in \widetilde M $ and $x=\pi(p)\in M$,
and preserves the lengths of horizontal vectors (see \cite{oneill}).
The structure of a Riemannian submersion gives a natural notion of vertical and horizontal vector fields.
Given $x\in M$,  we denote by $\mathcal F_x = \pi^{-1}(x)$ the \emph{fiber} of $\pi $ over $x$, which is
a submanifold of dimension $k=m-n$. 
We say that a vector field on $\widetilde M$ is  {\em vertical} if it is always tangent
to fibers, and it is called {\em horizontal} if it is always orthogonal to
fibers. On the other hand, the Riemannian metric on $\widetilde M$ give us the decomposition
of a vector field $X\in T \widetilde M$ into its vertical and horizontal components, say 
$X = X^{\mathcal V}+X^{\mathcal H}$.  Using these notations, the \emph{second fundamental form} of a fiber 
$\alpha ^\mathcal F: T \mathcal F \times T \mathcal F \to T^\perp \mathcal F$ is given by
 \[
 \alpha^\mathcal F(V, V) = (\widetilde \nabla_V V)^{\mathcal H},
 \]
where $\widetilde \nabla$ denotes the Levi-Civita conextion on $\widetilde M$.  
 The \emph{mean curvature} $H^\mathcal F$ of $\mathcal F$ is defined as the trace of the 
 second fundamental form, that is, it is the horizontal vector field given by
 \[
 H^\mathcal F =  \sum_{i=1}^{k}\alpha^\mathcal F(e_i,e_i)
 = \sum_{i=1}^{k}(\widetilde \nabla_{e_i}e_i)^\mathcal H,
 \]
where $\{e_1,\ldots,e_{k}\}$ is a local orthonormal frame to the fiber.
Now, given a smooth function $B$ on $M$ we set $\widetilde B : \widetilde M\to \real$,
$\widetilde B(p) = B\circ \pi (p)$ the lift of $B$ to $\widetilde M$. From Lemma  3.2 in \cite{CM} we
have
\begin{eqnarray}\label{des}
\widetilde \Delta \widetilde  B=\Delta B+\langle \widetilde \nabla \widetilde B, H^\mathcal F\rangle.
\end{eqnarray}
%\[
%\widetilde \Delta \widetilde  B(p)=\Delta B(x)+\langle \nabla \widetilde B(p), H^\mathcal b(p)\rangle.
%\]

Let us assume the function  $B:M^n\to \real$ is such that $\| \nabla B\|=1$  and 
$\Delta B \geq b>0$. As we saw above, this the case when $K_M\leq \kappa^2<0 $ or when $M$
is given by the warped product in Corollary \ref{cor} if we choose $B$ as a Busemann function. 
Clearly, under these conditions we also have $\Delta_p B >b$. Finally, we note that the
the gradient of $\widetilde B$ is the horizontal 
lifting of the gradient of $B$, in particular $|\widetilde\nabla \widetilde B| = 1$ and  so
$\widetilde \Delta_p \widetilde B=\widetilde \Delta \widetilde B$. 
Therefore, assuming that $\|H^\mathcal F\| \leq \alpha$ we get from \ref{des} that
$
\widetilde \Delta_p \widetilde  B \geq b - \alpha
$
Applying Theorem \ref{teo1} we get
\[
\lambda_{1,p}(\widetilde M) \geq \frac{(b-\alpha)^p}{p^p}.
\]
This proves Theorem \ref{sub}.
%i.e.,
%\begin{eqnarray}\label{eq:RelGradiente}
%\nabla \widetilde B=\widetilde{\nabla B}.
%\end{eqnarray}

\bibliographystyle{amsplain}
\bibliography{biblio}
\end{document}